\documentclass[12pt]{article}

\usepackage[top=1in, bottom=1in, left=1in, right=1in]{geometry}
\usepackage{tikz}
\usetikzlibrary{calc,arrows.meta}

\usepackage{amsmath,amssymb,amsthm}
\usepackage{color}

\newcommand{\R}{\mathbb{R}}

\newcommand{\im}{\mathrm{i}}

\newcommand{\In}{\mathrm{in}}
\newcommand{\Sc}{\mathrm{sc}}
\newcommand{\loc}{\mathrm{loc}}

\newtheorem{lemma}{Lemma}[section]

\newtheorem{theorem}{Theorem}[section]
\newtheorem{definition}{Definition}
\newtheorem{proposition}{Proposition}[section]
\newtheorem{corollary}{Corollary}[section]
\newtheorem{remark}{Remark}[section]

\newcommand{\hh}[1]{{\color{blue} #1}}

\begin{document}

\title{Transmission Eigenvalues  and Non-scattering.}
\author{Fioralba Cakoni and Michael S. Vogelius \footnote{Department of Mathematics, Rutgers University, New Brunswick,  New Jersey,  USA 
 (fc292@math.rutgers.edu) and  (vogelius@math.rutgers.edu)}}
\date{\empty}
\maketitle

 \begin{abstract}
 \noindent
In this paper we survey some recent results concerning scattering and non-scattering in the context of the linear Helmholtz equation and inhomogeneities of nontrivial contrast. We examine isotropic as well as anisotropic media. Part of the survey deals with the so-called transmission spectrum, namely those wave numbers at which non-scattering potentially may occur. For wave numbers that are not transmission eigenvalues any incident wave leads to scattering, however, being at a transmission eigenvalue is far from sufficient to guarantee the occurence of non-scattering for even a single incident wave. For instance the inhomogeneity generically has to be smooth for non-scattering to occur. Similarly many smooth geometric shapes will be scattering for natural incident waves even at a transmission eigenvalue. Part of the survey discusses recent results of that nature. 
 \end{abstract}


\section{Introduction} Scattering theory studies the effect that an inhomogeneity, viewed as a perturbation of a known background, has on an incoming wave. This effect is referred to as the scattered field, and we say that the incoming wave is scattered by the inhomogeneity if it is nonzero. Inverse scattering relies on the fact that the scattered wave carries information about the inhomogeneity and can therefore be used to image it. A natural question in this context is whether there exist incoming time-harmonic waves, at certain frequencies, that are not scattered by a given inhomogeneity, in other words, the inhomogeneity is invisible to probing by such waves. We refer to wave numbers corresponding to these frequencies, for which there exists a non-scattering incoming wave, as {\it{non-scattering wave numbers}}. We emphasize that, in this definition, a non-scattering wave number is associated with particular incoming waves. An inhomogeneous medium  that admits non-scattering wave numbers is called {\it a non-scattering inhomogeneity}. The existence of non-scattering wave numbers, or the lack thereof, is a fundamental question in inverse scattering and wave-based imaging. To formulate this question precisely, we next introduce our model scattering problem.

\smallskip
\noindent
Let  $n$ be a strictly positive scalar function in $L^{\infty}(\R^d)$, and $A=(a_{ij})$ be an $d\times d$ symmetric matrix-valued function  with $L^{\infty}(\R^d)$ entries satisfying 
\begin{equation}\label{eq:Aellip}
c_0^{-1}|\xi|^2 \le\xi^{\top} A(x) \xi\le c_0\,|\xi|^2\qquad\mbox{for almost  all $x\in\mathbb{R}^d$ and all $\xi\in\mathbb{R}^d$}
\end{equation}
for some positive constant $c_0$. We further assume that $A-I$ and $n-1$  are supported in $\overline{D} \subset {\mathbb R}^d$, $d=2,3$ where $D$  is a bounded region with Lipschitz boundary $\partial D$. In other words
$$
A= \begin{cases} A_D \hbox{ in } D \\
I  \hbox{ in } \Omega \setminus D \end{cases} \hbox{ and } n= \begin{cases} n_D \hbox{ in } D \\
1  \hbox{ in } \Omega \setminus D \end{cases}~.
$$
We denote by $\nu$ the outward unit normal vector defined almost everywhere on $\partial D$. In our model, $A_D(x),n_D(x)$ for $x\in D$ characterize the constitutive material properties of a (potentially)  anisotropic dielectric medium occupying the region $D$ situated in an isotropic homogenous background with constitutive material properties scaled to one.  In what follows we refer to the inhomogeneous medium defined above as $(A_D,n_D,D)$. 

\noindent
Consider now a time harmonic  interrogating wave $e^{-i\omega t} u^{\In}(x)$,  where the space-dependent part  $u^{\In}$ satisfies the Helmholtz equation 
\begin{equation}\label{helm}
\Delta u^{\In}+k^2u^{\In}=0,
\end{equation}
and where $k>0$ is the wave number proportional to the interrogating frequency $\omega$. This models wave propagation  in the background. Frequently, $u^{\In}$ may be taken as an entire solution of the Helmholtz equation in  $\R^d$,  such as an incident plane wave $e^{ikx\cdot \eta}$ propagating in the direction $\eta\in \mathcal{S}^{d-1}$, where $\mathcal{S}^{d-1}$ is the unit sphere in $\R^d$, or a superposition of plane waves, referred to as a Herglotz wave functions, given by
\begin{equation}\label{herg}
v_{\varphi}(x)=\int_{\mathcal{S}^{d-1}} \varphi(\eta)e^{ikx\cdot \eta}\,ds_{\eta}, \qquad \mbox{with density  $\varphi\in L^2(\mathcal{S}^{\eta-1})$.}
\end{equation}
However, the incident wave may also have singularities located in the exterior of the inhomogeneity $D$, such as point sources, single- or double-layer dipoles.  In general,  $u^{\In}$ satisfies the Helmholtz equation  (\ref{helm}) in a region  ${\mathcal O}$ which compactly contains the support of the inhomogeneity $D$. (see e.g. \cite{coltonkress}). The scattered field $u^{\Sc}$ in $\R^d\setminus D$ and the total field $u$ in $D$,  due to the  incident field $u^{\In}$,  satisfy
\begin{eqnarray}
&\Delta u^{\Sc} + k^2  u^{\Sc}=0& \quad\mbox{in $~\R^d \setminus \overline{D}$} \label{maineq1}\\
&\nabla\cdot A_D \nabla u+ k^2 n_D  u= 0& \quad\mbox{in $D$}\label{maineq2}\\
&u=u^{\Sc}+u^{\In}\quad \mbox{and}  \quad \nu^{\top} \cdot A_D \nabla u=\partial_\nu(u^{\Sc}+u^{\In})&\quad \mbox{on $\partial D$}\label{maineq3}\\
&\lim\limits_{|x|\to\infty} |x|^{\frac{d-1}{2}}\left(\frac{\partial }{\partial |x|}u^{\Sc}-\im ku^{\Sc}\right)=0& ~\mbox{ uniformly for $\hat{x}:=\frac{x}{|x|}\in \mathcal{S}^{d-1}$}. \label{maineq4}
\end{eqnarray}
Obviously, the equations (\ref{maineq1})-(\ref{maineq2}) together with the transmission conditions  (\ref{maineq3}) can be equivalently written as an equation in the whole space  for the scattered field $u^{\Sc}$ (where in $D$ we set $u^{\Sc}:=u-u^{\In}$) 
\begin{equation}\label{MainGov1}
\nabla\cdot A \nabla u^{\Sc} + k^2 n  u^{\Sc}= \nabla\cdot (I-A)\nabla u^{\In}+k^2(1-n)u^{\In} \qquad\mbox{in $~\R^d$}\\
\end{equation}
along with the Sommerfeld radiation condition (\ref{maineq4}). We  recall that $(I-A)$ and $1-n$ are supported in $D$. It is known  that,  if in addition the entries of $A_D$ are in  $W^{1,\infty}(D)$ (this assumption is needed for the unique continuation) then (\ref{MainGov1})  admits a \emph{unique  solution} $u^{\Sc}\in H^1_{\loc}(\R^d)$ (see e.g. \cite{CakoniColtonHaddar2016}).  In the case of $A \equiv I$ included here, this solution  is $H^2_{\loc}(\R^d)$.  Additional regularity conditions on the coefficients $A_D,n_D$ and  the inhomogeneity $D$ will be imposed later as required. Of main interest to us is the case when $A$ or $n$ has a jump across $\partial D$.  

\noindent
In the framework of this model the  incident wave  $u^{\In}$, at wave number $k$, is not scattered by inhomogeneity  $(A_D,n_D,D)$ exactly when
\begin{eqnarray}\label{nons1}
&&\nabla\cdot A \nabla u^{\Sc} + k^2 n  u^{\Sc}= \nabla\cdot (I-A)\nabla u^{\In}+k^2(1-n)u^{\In}  \quad \mbox{in $\R^d$} \nonumber \\
&&\hbox{ with }\qquad u^{\Sc}\equiv 0 \quad \mbox{in $\R^d\setminus \overline{D}$},
\end{eqnarray}
in other words, $u^{\Sc}$ is  compactly supported with support inside $\overline{D}$. Equivalently, with the assumption of a Lipschitz smooth $\partial D$, this means $u^{\Sc}$ solves the overdetermined problem
\begin{eqnarray}
&\nabla\cdot A_D \nabla u^{\Sc} + k^2 n_D  u^{\Sc}= \nabla\cdot (I-A_D)\nabla u^{\In}+k^2(1-n_D)u^{\In}&  \quad \mbox{in $D$}~, \label{nons21}\\
&u^{\Sc}=0\quad \mbox{and}  \quad \nu^{\top} \cdot A_D \nabla u^{\Sc}=\nu^{\top} \cdot (I-A_D) \nabla u^{\In}&\quad \mbox{on $\partial D$}~.\label{nons22}
\end{eqnarray}

\begin{definition}[Non-scattering Inhomogeneity] {\em We say that the inhomogeneity $(A_D,n_D,D)$ is non-scattering if there exists a wave number $k>0$ and a nontrivial incident wave $u^{\In}$ such that (\ref{nons1}) or equivalently (\ref{nons21})-(\ref{nons22}) has a solution.
This $k>0$ is called a non-scattering wave number with the corresponding non-scattering incident wave $u^{\In}$.}
\end{definition}
\noindent
As already mentioned, for a given inhomogeneity and a given type of incident wave, to find non-scattering wave numbers amounts  to solving an overdetermined problem. To determine  if non-scattering occurs, we relax the requirement by including the incident field as part of the unknowns. More precisely, we consider $v:=u^{\In}|_D$ as an unknown that satisfies  the  Helmholtz equation  in $D$, and obtain the following homogeneous boundary value problem for two elliptic equations defined only in $D$
\begin{equation}\label{te}\left\{
\begin{array}{rrrrclll}
&\nabla\cdot A_D \nabla u^{\Sc} + k^2 n_D  u^{\Sc}= \nabla\cdot (I-A_D)\nabla v+k^2(1-n_D)v&  &\mbox{in \; \;\quad $D$}\\
&\Delta v+k^2v=0\qquad \qquad \qquad \qquad &   &\mbox{in \; \;\quad $D$} \\
&u^{\Sc}=0\quad \mbox{and}  \quad \nu^{\top} \cdot A_D \nabla u^{\Sc}=\nu^{\top} \cdot (I-A_D) \nabla v\;\; \;\;\; \;\; &  & \mbox{on \; $\partial D$}~.
\end{array}
\right.
\end{equation}
This is  known as the transmission eigenvalue problem. It is a non-selfadjoint eigenvalue problem with  challenging mathematical structure.  Rewritten for $v$ and $u=v+u^{\Sc}$ the {\it{transmission eigenvalue problem}} reads: find $k$, $v\neq 0$ and $u\neq 0$ such that
\begin{equation}\label{te2}\left\{
\begin{array}{rrrrclll}
&\nabla\cdot A_D \nabla u + k^2 n_D  u= 0&  \; \mbox{in \; \quad $D$} \\
&\Delta v+k^2v=0 \qquad &  \; \mbox{in \, \quad $D$} \\
&u=v\qquad \quad &  \;\mbox{on \quad $\partial D$} \\
&\nu^{\top} \cdot A_D \nabla u=\nu^{\top} \nabla v \quad &  \; \mbox{on  \quad $\partial D$}~.
\end{array}
\right.
\end{equation}
\begin{definition}[Transmission Eigenvalues] \em{Values of $k\in {\mathbb C}$ for which (\ref{te2}) admits a nontrivial solution $(u,v)$ are called transmission eigenvalues. The non-trivial solution $(u,v)$ is referred to as the  corresponding eigenfunction.}
\end{definition}
\noindent
The transmission eigenvalue problem having a solution is a necessary condition for an inhomogeneity to be non-scattering. If $k>0$ is a non-scattering wave number with incident field $u^{\In}$, then $k$ is a transmission eigenvalue with corresponding  eigenvector $v:=u^{\In}|_D$ and $u:=v+u^{\Sc}$. The converse is not generally true.  If $k>0$ is a real transmission eigenvalue, the $v$ part of the eigenvector  is defined  only in $D$ and is not necessary the restriction of an incident field which solves the Helmholtz equation in a larger region ${\mathcal O}\supseteq \overline{D}$.  As we will describe latter, the existence of  infinitely many real transmission eigenvalues is proven for a large class of (not necessarily regular) inhomogeneities. A central question therefore becomes under what circumstances a transmission eigenvalue yields a non-scattering inhomogeneity. It turns out that the existence of a non-scattering wave number (unlike the existence of real transmission eigenvalues) generically implies a certain regularity of the inhomogeneity; at the core it becomes a question of free boundary regularity. For very regular (real analytic) inhomogeneities the natural question now arises whether (or when) the incident part, $v$, of the transmission eigenfunction can be extended to all of $\R^d$, or whether it can take the form of one of the special incident waves mentioned earlier.

\noindent 
The analysis for $A\equiv I$ and $A\not\equiv I$ differ fundamentally. Our discussion considers the case $A\equiv I$ in the next section and the case $A\not\equiv I$ in Section~\ref{AneqI}. The latter case exhibits a richer structure.  We conclude this introduction by noting that the transmission eigenvalue problem, introduced independently in 1986 in \cite{CM} and \cite{kirsch}, has since attracted considerable interest from both the inverse scattering and spectral theory communities, as evidenced by a vast and growing literature. The proper understanding of the transmission eigenvalue problem is a fundamental component to establish uniqueness results for various inverse problems and serves as a key building block in many reconstruction methods for inverse scattering problems. Furthermore, transmission eigenvalues can in principle be determined from scattering data, and thus allow for the extraction of information about the constitutive material properties of inhomogeneities. We refer the reader to the monograph \cite{CakoniColtonHaddar2016} for an up-to-date discussion and a comprehensive list of references on the subject. The existence of non-scattering inhomogeneities was first addressed in \cite{nonscat4} in 2014 and in subsequent works \cite{ nn5, ElH18, HSV16, PSV17}, where it was shown that inhomogeneities with corners, edges, or conical singularities will be scattering. The connection between non-scattering and free boundary problems was first explored in \cite{CV} and \cite{fb7}, offering a new perspective that led to deeper results in \cite{CVX, fb2, fb3, fb4}, which broadly speaking demonstrate that almost all singularities scatter. The case of analytic inhomogeneities has only recently been studied in \cite{HV} and \cite{HV2}, using generalizations of techniques that were previously used for the analysis of the so-called Pompeiu/Schiffer problem (in \cite{GaSe} and \cite{BrKa}). In addition to its mathematical significance, the question of whether non-scattering inhomogeneities exist is also important for applications. In particular, at a non-scattering wave number, the relative scattering operator, also known as the far-field operator, is not injective (see \cite{res1} and \cite{coltonkress}). This non-injectivity may lead to the failure of certain reconstruction methods in inverse scattering.
 
\section{The Case of $A\equiv I$.}
In this case the inhomogeneity, now denoted by $(n_D,D)$, has contrast only in the lower order terms. We always assume that $\partial D$ is Lipschitz and  $n_D\in L^\infty(D)$, $n_D>0$.  Additional assumptions are added as needed.  In this case, the inhomogeneity $(n_D,D)$ is non-scattering if there exists a $k>0$ and $v\neq 0$  satisfying
\begin{equation}\label{nonsn1}
\Delta v+k^2v=0  \quad \mbox{in ${\mathcal O}\supset \overline{D}$}~, 
\end{equation}
such that the following problem has solution 
\begin{eqnarray}
&\Delta w + k^2 n_D  w= k^2(1-n_D)v&  \quad \mbox{in $D$}~, \label{nonsn2}\\
&w=0\quad \mbox{and}  \quad \displaystyle{\frac{\partial w}{\partial \nu}}=0 &\quad \mbox{on $\partial D$}~,\label{nonsn3}
\end{eqnarray}
where for simplicity we denote $w:=u^{\Sc}$ and $v:=u^{\In}$, the corresponding non-scattering incident wave. The transmission eigenvalue problem reads: Find  nonzero $w\in H^2_0(D)$ and $v\in L^2(D)$ such that
\begin{equation}\label{ten}
\left\{
\begin{array}{rrrccllll}
&\Delta w + k^2 n_D  w= k^2(1-n_D)v& \qquad  \mbox{in} & D&~,\\
&\Delta v+k^2v=0 \qquad \; & \qquad  \mbox{in} & D &~,\\
&w=0\quad \mbox{and}  \quad \displaystyle{\frac{\partial w}{\partial \nu}}=0 &\mbox{on} &  \partial D~.&
\end{array}
\right.
\end{equation}
If we set  $u:=w+v$, then (\ref{ten}) can be written as: Find nonzero $u\in L^2(D)$ and $v\in L^2(D)$, with $u-v\in H^2_0(D)$, such that 
\begin{equation}\label{ten2}
\left\{
\begin{array}{rrrccllll}
&\Delta u + k^2 n_D  u= 0& \qquad  \mbox{in} & D&~,\\
&\Delta v+k^2v=0  & \qquad  \mbox{in} & D&~, \\
&u-v=0 & \qquad  \mbox{on} &  \partial D&~,\\
&\displaystyle{\frac{\partial u}{\partial \nu}}-\displaystyle{\frac{\partial v}{\partial \nu}}=0 & \qquad  \mbox{on} &  \partial D~,&
\end{array}
\right.
\end{equation}
where the equations for $u$ and $v$ hold in the distributional sense and  $H^2_0(D)$ is the space of functions in $H^2(D)$ with zero Cauchy data $w=\frac{\partial w}{\partial \nu}=0$ on $\partial D$. 
\subsection{Spherically symmetric inhomogeneities} \label{sphere}
To shed light onthe structure of non-scattering inhomogeneities consider the spherically symmetric case where $D=B_1(0):=\left\{x\in {\mathbb R}^3:\, |x|<1\right\}$,  and $n_{B_1(0)}(x):=n(r)$ is a function of the radial variable only and furthermore $n\in C^2[0,1]$. Introducing spherical coordinates $(r,\hat x)$, we look for solutions of the equations in  (\ref{ten2})  in the form
$$
v(r,\hat x)=j_{\ell}(kr)Y_\ell(\hat x)  \qquad \mbox{and}\qquad u(r,\hat x)=b_{\ell,k}y_{\ell,k}(r)Y_\ell(\hat x), \qquad \ell=0,1,\ldots
$$
where $Y_\ell$  denotes one of the $2\ell+1$ linearly independent spherical harmonics of order $\ell$ (these form a complete orthogonal system in $L^2({\mathcal S}^2)$, see e.g. \cite[Theorem 2.8]{coltonkress}), $j_\ell$ is a spherical Bessel function, $b_{\ell,k}$ is a constant, and $y_{\ell,k}$ is a solution to
$$y''_{\ell,k}+\frac{2}{r}y'_{\ell,k}+\left(k^2n(r)-\frac{\ell(\ell+1)}{r^2}\right)y_{\ell,k}=0 \hbox{ for } r>0~,$$ 
which behaves like $j_\ell(kr)$ as $r\to 0$, {\it i.e.}, 
$$\lim_{r\to 0}r^{-\ell}y_{\ell,k}(r)=\frac{\sqrt{\pi}k^\ell}{2^{\ell+1}\Gamma(\ell +3/2)}~.$$
Applying the boundary conditions in (\ref{ten2}) for $r=1$, we can deduce that $k\in {\mathbb C}$  is a  transmission eigenvalue if and only if
\begin{equation}\label{e2.2}
d_\ell(k)=\mbox{det}\left(\begin{array}{rrcll}
y_{\ell,k}(1) & -j_\ell(k)\\
 &\\
y'_{\ell,k}(1) & -kj'_\ell(k)
\end{array}\right)=0.
\end{equation}
The first observation is that $d_\ell(k)$, $\ell=0,1\cdots $ are entire functions of $k$ and that $d_\ell(k)\equiv 0$  if and only if $n(r)\equiv 1$ \cite[Theorem 6.1]{CakoniColtonHaddar2016}. Thus for nonzero contrast,  the zeros  of each $d_\ell(k)$,  in the complex plane are discrete without interior accumulation points.  In the special case $n(1)=1$, it is known that  $d_\ell(k)$ has the following asymptotic behavior as  $k>0$ goes to $+\infty$
\begin{equation}\label{e2.3}
d_\ell(k)=\frac{1}{k\left[n(0)\right]^{\ell/2+1/4}}\,\sin k\left(1-\int_0^1[{n(r)}]^{1/2}dr\right)+O\left(\frac{\ln k}{k^2}\right),
\end{equation}
which immediately shows that $d_\ell(k)$ has infinitely many real zeros (real transmission eigenvalues) provided 
\begin{equation}\label{cond}
M_n:=\int_0^1[{n(r)}]^{1/2}dr\neq 1.
\end{equation}
A simimilar formula holds for $n(1)\neq 1$, which may be used to show that in this case there are also infinitely many zeros, even when $M_n=1$. 
Each transmission eigenvalue corresponding to a zero of $d_\ell(k)$ has multiplicity at least $2\ell + 1$, and  this multiplicity is finite, as the transmission eigenvalues are  eigenvalues of a compact operator and thus have finite-dimensional eigenspaces.  For each real transmission eigenvalue, the part  $v$ of the eigenfunction, $v_\ell(r,\hat x)=j_{\ell}(kr)Y_\ell(\hat x)$, is a solution of the Helmholtz equation in ${\mathbb R}^3$. In fact  by the Funk - Hecke formula (see \cite[Chapter 2]{coltonkress}), $v_\ell$ is a Herglotz wave function (see (\ref{herg})). Hence every real transmission eigenvalue is a non-scattering wave number. The following proposition summarizes the above discussion.
\begin{proposition}\label{snon}
Spherically symmetric inhomogeneities $(B_1(0), n(r))$, with $n(1)\neq 1$ or $M_n\neq 1$,  are non-scattering. For such inhomogeneities,  the set of non-scattering wave numbers coincides with the set of real transmission eigenvalues. More specifically  $(B_1(0), n(r))$ admits an infinite family of non-scattering incident  Herglotz wave  functions (\ref{herg}) of the form $v_\ell(r,\hat x)=j_{\ell}(kr)Y_\ell(\hat x)$, $\ell=0,1, \cdots$, where each incident wave is not scattered at infinitely many  non-scattering wave numbers $k>0$, which are zeros of $d_{\ell}(k)$. 
\end{proposition}
\noindent
We remark that the spherically symmetric inhomogeneities are the only ones known to be non-scattering in the case $A=I$. We return to this issue in Section \ref{anal}.

\noindent
Although in the context of non-scattering only real transmission eigenvalues are relevant, it is nevertheless desirable  to investigate whether  truly complex eigenvalues exist. In the spherically symmetric case, this amounts to determining whether  the determinants $d_\ell(k)$, which are entire functions of $k$ of order one,  have truly complex zeros. In \cite{CY2} it is proven that, under additional assumptions $n\in C^3[0,\,1]$ such that $n'''$ is absolutely continuous, $n(1) = 1$,  $n'(1) = 0$, and $n''(1)\neq 0$, the function  $d_0(k)$ has infinitely many truly complex zeros as well. This particular case demonstrates that the {\it{transmission eigenvalue problem is non-selfadjoint}}. The proof is based on comparing  the density of the positive zeros of $d_0(k)$, obtained via the asymptotic expansion, to the density of its zeros in right half-plane using the Cartwright - Levinson Theorem.  

\noindent
Finally, for the reader interested in inverse spectral problems, we mention that the refractive index $n(r)$ is uniquely determined from the knowledge of all real and complex transmission eigenvalues, provided that  either $n(r)\geq 1$ or $0<n(r)\leq 1$ and the value $n(0)$ is known. If one knows a priori that $0<n(r)<1$, then it can be shown that $n(r)$ is uniquely determined from the knowledge of all the zeros of $d_\ell(k)$ for a fixed $\ell$, without requiring knowledge of $n(0)$. For more details, we refer the reader to \cite{CakoniColtonHaddar2016}.
\subsection{The transmission eigenvalue problem}\label{TEn}
We return our attention to the transmission eigenvalue problem (\ref{ten}) or (\ref{ten2}).  This eigenvalue problem has a deceptively simple formulation, namely two elliptic PDEs in the  bounded domain $D$ with a single set of Cauchy data on the boundary, but as we have already seen in the spherically symmetric case, it is a non-selfadjoint eigenvalue problem. Since the problem is formulated entirely on $D$, we drop, for simplicity of notation, the subscript on the index of refraction and refer to it as $n$. We assume that $n-1$ is uniformly of one sign: to fix our ideas  we assume that  $\inf_D  n=n_{*}>1$. Dividing by $n-1$ in (\ref{ten}) the problem becomes to find  nonzero $u\in H^2_0(D)$  solving
\begin{equation}\label{forthorder}
\left(\Delta+k^2\right)\frac{1}{n-1}\left(\Delta +k^2n\right)u=0  \quad \mbox{in } D~,
\end{equation}
and this may be written as
\begin{equation}\label{peno}
({\mathbb T}+\, k^2 {\mathbb T}_1+ k^4 {\mathbb T}_2)u=0~,
\end{equation}
where the bounded linear operators ${\mathbb T}, {\mathbb T}_1, {\mathbb T}_2 \!:\!H^2_0(D)\to H^2_0(D)$ are defined by means of the Riesz representation  theorem as follows
$$\left({\mathbb T}u, \psi\right)_{H^2(D)}=\int_D\frac{1}{n-1}\Delta u \, \Delta \overline{\psi} \, \textrm{d}x~,  \qquad \qquad  \left({\mathbb T}_2 u, \psi\right)_{H^2(D)} = \int_D\frac{n}{n-1} u \,  \overline{\psi} \, dx~,$$
$$\begin{array}{rrclll}
&\left({\mathbb T}_1u, {\psi}\right)_{H^2(D)}=\displaystyle{\int_D\frac{n}{n-1}u\ \Delta  \overline{\psi} \,dx+\int_D \frac{1}{n-1}\Delta u \,   \overline{\psi}\, dx}&\\
 &\quad =\displaystyle{\int_D\frac{1}{n-1}\left(\Delta u\,  \overline{\psi} + u\ \Delta  \overline{\psi} \right)\,dx-\int_D \nabla u \cdot \nabla  \overline{\psi}\, dx}& \end{array}
 $$
for all $u,\psi \in H^2_0(D)$. Obviously all these operators are selfadjoint; in addition ${\mathbb T}$ is positive invertible, ${\mathbb T}_1$ is compact  and ${\mathbb T}_2$ is compact and non-negative (compactness holds thanks to the compact embedding of $H^2(D)$ into $L^2(D)$ and $H^1(D)$). We introduce the operators ${\mathbb K}_1=-{\mathbb T}^{-1/2}{\mathbb T}_1 {\mathbb T}^{-1/2}$ and ${\mathbb K}_2={\mathbb T}^{-1/2}{\mathbb T}_2 {\mathbb T}^{-1/2}$. It is now convenient to rewrite  (\ref{peno}) as a more classical eigenvalue problem 
$$\left({\mathbf K}- \frac{1}{k^2} {\mathbf I}\right)U=0, \qquad  U\in H^2_0(D)\times H^2_0(D)$$
for the compact matrix operator ${\mathbf K}: H^2_0(D)\times H^2_0(D) \to H^2_0(D)\times H^2_0(D)$ given by
$${\mathbf K}:=\left(\begin{array}{rrcll}
 {\mathbb K}_1\;\;\;&-{\mathbb K}_2^{1/2}\\
{\mathbb K}_2^{1/2}& 0\;\;\;\;
\end{array}
\right)~,
$$
with   $U:=\left({\mathbb T}^{1/2}u, k^2 {\mathbb K}_2^{1/2} {\mathbb T}^{1/2}u\right)$ (here we also use that ${\mathbb K}_2$ is non-negative).  Although ${\mathbb K}_1$ and ${\mathbb K}_2$ are selfadjoint, ${\mathbf K}$ is not, which again reveals the non-selfadjoint nature of the transmission eigenvalue problem. This analysis leads to
\begin{theorem}
Assume $\partial D$ is Lipschitz,  $n\in L^\infty(D)$ and $n_{*} = \inf_D  n>1$. Then the set of transmission eigenvalues $k\in {\mathbb C}$ is at most discrete with $\infty$ as the only possible accumulation point. The corresponding eigenspaces are of finite dimension.
\end{theorem}
\noindent
Non-scattering wave numbers are a subset of the real transmission eigenvalues, so a natural question arises: ``do real transmission eigenvalues exist". Following \cite{1}, after multiplitation of (\ref{forthorder}) by $\psi$ and integration by parts, we  define  the bounded linear operators ${\mathbb A}_\tau: H^2_0(D)\to H^2_0(D)$ and  ${\mathbb B}:H^2_0(D)\to H^2_0(D)$ by means of the Riesz representation  theorem 
\begin{eqnarray*}
\left({\mathbb A}_\tau u,\psi\right)_{H^2(D)}&=&\int_D\frac{1}{n-1}(\Delta  u+\tau u) (\Delta   \psi+\tau \psi)\,dx +\tau^2 \int_D u\psi\,dx~,\\
 \left({\mathbb B} u,\psi\right)_{H^2(D)}&=&\int_D\nabla u \cdot \nabla \psi\, dx~,
\end{eqnarray*}
with $\tau:=k^2$.   Thus the transmission eigenvalue problem reads
\begin{equation}\label{tep}
({\mathbb A}_\tau-\tau {\mathbb B})u=0~.
\end{equation}
Obviously ${\mathbb B}$ is a compact and positive linear operator. Simple estimates give that 
$$
\left({\mathbb A}_\tau u,u\right)_{H^2(D)}\geq\gamma \|\Delta u+\tau u\|_{L^2}^2+\tau^2\|u\|^2_{L^2}\geq\left(\gamma-\frac{\gamma^2}{\sigma}\right)\|\Delta u\|_{L^2(D)}^2+ (1+\gamma-\sigma)\tau^2\|u\|_{L^2}^2~,\nonumber\\
$$
where $\gamma:=\frac{1}{n^{*}-1}$, $n^*= \sup_D n$ and $\sigma$ is any number with $\gamma<\sigma<\gamma+1$. This implies coercivity since the $L^2$-norm of $\Delta u$ is equivalent of $H^2$-norm of $u$ in $H^2_0(D)$. Therefore $\tau\in (0,\,+\infty)\mapsto {\mathbb A}_\tau \in {\mathcal L}(H^2_0(D))$ is a continuous mapping into the space of self-adjoint positive definite bounded linear operators.  
\begin{proposition}
 If $n_{*}>1$ then values of $k>0$ for which  $k^2<\frac{\lambda_1(D)}{n^{*}}$, where $\lambda_1(D)$ is the first Dirichlet eigenvalue of $-\Delta$ in $D$, cannot be transmission eigenvalues (or non-scattering wave numbers).
\end{proposition}\label{prop1}
\begin{proof} From above we have
$$\left({\mathbb A}_\tau u-\tau {\mathbb B}u,\,u\right)_{H^2_0} \geq\left(\gamma-\frac{\gamma^2}{\sigma}\right)\|\Delta u\|_{L^2}^2+ (1+\gamma-\sigma)\|u\|_{L^2}^2-\tau\|\nabla u\|^2_{L^2}~.$$
Using the Poincar\'e inequality  for $\nabla u\in H^1_0(D)$ we get $\|\nabla u\|^2_{L^2(D)}\leq \frac{1}{\lambda_1(D)} \|\Delta u\|^2_{L^2(D)}$, and so
$$
\left({\mathbb A}_\tau u-\tau {\mathbb B}u,\,u\right)_{H^2} \geq \left(\gamma-\frac{\gamma^2}{\sigma}- \frac{\tau}{\lambda_1(D)}\right)\|\Delta u\|_{L^2}^2+ \tau(1+\gamma-\sigma)\|u\|_{L^2}^2.\nonumber
$$
Hence  ${\mathbb A}_\tau -\tau {\mathbb B}$ is positive as long as
$\tau<(\gamma-\frac{\gamma^2}{\sigma})\lambda_1(D)$. We recall that $\gamma =\frac{1}{n^{*}-1}$;  a choice of $\sigma$ arbitrarily close to $\gamma+1$ therefore results in the requirement that $\tau<\frac{\gamma}{1+\gamma}\lambda_1(D) = \frac{\lambda_1(D)}{n^{*}}$.
\end{proof}
\noindent
The formulation (\ref{tep}) suggests to consider the generalized eigenvalue problem 
\begin{equation}\label{tep2}
({\mathbb A}_\tau-\lambda(\tau) {\mathbb B})u=0~,
\end{equation}
which for fixed $\tau>0$ is known to have an infinite sequence of eigenvalues $\lambda_j>0$ accumulating at $+\infty$. For  fixed $\tau>0$ these satisfy the Courant-Fischer min-max principle (see e.g. \cite[Section 4.1]{CakoniColtonHaddar2016})
$$\lambda_j(\tau) = \min_{W \in {\mathcal U}_j} \max_{u \in W \setminus \{0\}}
   \frac{({\mathbb A}_\tau u, u)}{({\mathbb B} u, u)}~,$$
where ${\mathcal U}_j$ denotes the set of all $j$-dimensional subspaces $W$ of $H^2_0(D)$.  In particular, each $\lambda_j(\tau)$ for $j\in {\mathbb N}$ is a continuous function of $\tau$. A value $\tau>0$  corresponds to a transmission eigenvalue if it solves $$\lambda_j(\tau)-\tau=0~.$$
Proposition \ref{prop1} implies that
$$\lambda_j(\tau_0)-\tau_0>0, \qquad \mbox{for any $0<\tau_0<\frac{\lambda_1(D)}{n^{*}}$ and all $j\in {\mathbb N}$~.}$$ If we show that ${\mathbb A}_{\tau_1} - \tau_1 {\mathbb B}$ is non-positive on a $\ell$-dimensional subspace ${\mathcal W}$ of $H^2_0(D)$, which means that $\lambda_j(\tau_1)-\tau_1\le 0$, $j={1 \dots \ell}$,  then we can conclude that each $\lambda_j(\tau)-\tau$ has at least one zero in $(\tau_0, \tau_1]$, hence  there are at least $\ell$ real transmission eigenvalues $k>0$ (counting multiplicity) such that  $\tau_0<k^2\leq \tau_1$.  To show this, let  $k_{1,n_{*}}$ be the first transmission eigenvalue for the unit ball $B_1$ and  $n(r):=n_{*}$ constant. By a scaling $k_{\epsilon,n_{*}}:=k_{1,n_{*}}/\epsilon$ is the first transmission eigenvalue the ball of radius $\epsilon>0$. Given an integer $\ell\ge 1$ choose $\epsilon(\ell)>0$ sufficiently small  that $D$ compactly contains
$\ell$  disjoint balls $B_{\epsilon}^i$, $i=1\dots \ell$ of radius $\epsilon$.  Consider the transmission eigenfunction  $u^{B^i_\epsilon,n_{*}}\in H^2_0(B^i_\epsilon)$  of each ball  corresponding to the same eigenvalue $k_{\epsilon,n_{*}}$. The extension by zero $\tilde u^i$
of $u^{B^i_\epsilon,n_{*}}$ to the whole of $D$ is obviously in $H^2_0(D)$, due to zero Cauchy data on $\partial B^i_\epsilon$.  The set  $\{\tilde u^1, \tilde u^2, \dots, \tilde u^\ell\}$ is linearly independent and orthogonal in $H^2_0(D)$ since $\tilde u^i$  have disjoint supports, and we denote by ${\mathcal U}$ the $\ell$-dimensional subspace of $H^2_0(D)$ spanned by this set. For $\tau_1:=k_{\epsilon,n_{*}}^2$ and for every $\tilde u\in {\mathcal U}$
\begin{eqnarray*}
\left({\mathbb A}_{\tau_1} \tilde u-\tau_1 {\mathbb B} \tilde u,\,\tilde u\right)_{H^2(D)} &=&\int\limits_{D}\frac{1}{n-1}|\Delta  \tilde u+\tau_1 \tilde u|^2\, dx +\tau_1^2\int\limits_{D}|\tilde u|^2\,dx-\tau_1\int\limits_D|\nabla \tilde u|^2\,dx\nonumber\\
&&\hspace*{-3cm}\leq \int\limits_{D}\frac{1}{n_{*}-1}|\Delta  \tilde u+\tau_1 \tilde
u|^2\, dx+\tau_1^2\int\limits_{D}|\tilde
u|^2\,dx-\tau_1\int\limits_D|\nabla \tilde u|^2\,dx\!=\!0~,
\end{eqnarray*}
where we use the equation  satisfied by the eigenpairs $(k_{\epsilon,n_{*}}\,,\,u^{B^i_\epsilon,n_{*}})$, for $i=1,\dots, \ell$.  We note that 
and $k_{\epsilon,n_{*}}$  goes to $+\infty$ as $\epsilon \to 0$, and $\epsilon$ goes to zero as $\ell\rightarrow +\infty$. Since the multiplicity of each transmission eigenvalue is finite and since the only possible accumulation point is $\infty$, we have shown, by letting $\ell \to +\infty$,  that there exists an infinite countable set of real transmission eigenvalues that accumulate at $+\infty$. In other words: we have shown that for  $\partial D$ Lipschitz and  $n\in L^\infty(D)$ such that  $n_{*} = \inf_D  n>1$, there exists an infinite set of real  transmission eigenvalues $k>0$  that accumulate only at $+\infty$. 

\noindent
The assumption that the contrast $n-1$ is one sign uniformly in $D$ is not absolutely necessary. We observe that the transmission eigenvalue problem (\ref{ten2}) can be recast as finding the values of $k$ for which Kern$({\mathcal D}_{k,n}-{\mathcal D}_{k,1})$ is non-trivial, where  ${\mathcal D}_{k,n},{\mathcal D}_{k,1}$ are the Dirichlet-to-Neumann operators for $(n,D)$ and $(1,D)$ defined by
$${\mathcal D}_{k,q}: f\mapsto \frac{\partial u_f}{\partial \nu}, \qquad \mbox{where \qquad $\Delta u_f +k^2qu_f=0$  in $D$, \quad $u_f=f$ on $\partial D$}.$$
This suggests that the sign of $n-1$ in a neighborhood  of the boundary $\partial D$, is what matters. We summarize below the state-of-the-art results  in this regard. We introduce the boundary neighborhood  $D_\delta\subset D$, defined  by 
\begin{equation}\label{neib}
D_\delta:=\left\{x\in D, \; \mbox{dist}(x,\partial D)<\delta\right\}\;\; \;\;\; \mbox{for some fixed $\delta>0$}~,
\end{equation}
and
$$n_{\star}:= \inf_{D_\delta} n \qquad \mbox{and}\qquad n^{\star}:= \sup_{D_\delta}  n~.$$
\begin{enumerate}
\item If $\partial D$ is Lipschitz,  $n\in L^\infty(D)$ is positive, and either $n_{\star}>1$ or $n^{\star}<1$, for some neighborhood $D_\delta$, then the set of transmission eigenvalues in  ${\mathbb C}$  is discrete without interior accumulation points. Each eigenvalue has finite multiplicity and the resolvent of the transmission eigenvalue problem is Fredholm. For the proof of these results we refer the reader to \cite{kirsch2} \cite{syl}.
\item  Under stronger regularity assumptions, more precisely for  $\partial D$ of class $C^3$ and $n\in C^1(\overline D)$ such that $n(x)\neq 1$ for $x\in \partial D$, \cite{nguyen2} proves that the set of generalized transmission eigen-pairs $(u,v)$ is complete in $L^2(D)\times L^2(D)$. It is shown in  \cite{vodev} that all transmission eigenvalues lie in a strip around the real axis. These papers also establish Weyl  estimates for the counting function of transmission eigenvalues, showing that the number of all eigenvalues inside the disk of radius $r$ is asymptotically of order {\hh{$r^{d/2}$}} with an explicit constant depending only on $D,n$. 
\end{enumerate}
The analysis mentioned in the above two items refers to all transmission eigenvalues, real or complex. Again, in connection with non-scattering it is important to know when real  transmission eigenvalues exist. The approach described above to prove the existence of real  transmission eigenvalues was adapted  in \cite{CCH} to the case when $\partial D$ is Lipschitz,  $n\in L^\infty(D)$ is positive, either $n_{\star}>1$ or $n^{\star}<1$ in $D_\delta$, and $n\equiv 1$ in $D\setminus \overline{{D_\delta}}$. In this case, the fourth order variational formulation is kept in ${D_\delta}$, and the Helmholtz equation satisfied by $u:=w-v$ in $D\setminus \overline{D_\delta}$ is  enforced through the space on which the operators are defined.  More precisely, the transmission eigenvalue problem is formulated in variational form as finding $u\in V_0(D,D_\delta,k)$ such that 
$$\int_{{D_\delta}} \frac{1}{n-1}\left(\Delta u
 +k^2 u \right)\,\left(\Delta \bar \psi +k^2 \bar \psi \right)\, dx +k^2\int_{{D_\delta}} (\Delta u
+ k^2 u)\, \bar \psi\, dx = 0$$
for all $\psi \in V_0(D,D_0,k)$, where the Hilbert space $V_0(D,D_\delta,k)$ is defined by 
$$V_0(D,D_\delta,k):=\left\{u\in H^2_0(D) \quad \mbox{such that  $\Delta u+k^2u=0$ in $D\setminus \overline{D_\delta}$}\, \right\}.$$ 
Recently, this idea is generalized  in \cite{CH2} to prove the existence of real transmission eigenvalues without any assumption on the contrast $n-1$ in $D\setminus \overline{D_\delta}$. The following theorem states this most recent result on the existence of real transmission eigenvalues.
\begin{theorem}
Assume that $\partial D$ is Lipschitz,  $n\in L^\infty(D)$ is positive, and either $n_{\star}>1$ or $n^{\star}<1$, in some $D_\delta$, $\delta>0$. Then there exists a countable infinite set of real  transmission eigenvalues $k>0$  that accumulate only at $+\infty$. 
\end{theorem}
\noindent
A long standing open problem related to the transmission eigenvalue problem is to understand its spectrum  without any conditions on the contrast $n-1$ in $D$.  Section 7 in \cite{uhl} explains the connection of this question with the unique determination of the sound speed and initial source in photo-acoustic tomography.

\noindent
We conclude this section by addressing the question of why we do not consider a complex-valued refractive index, which models absorbing and dispersive media.
\begin{proposition} \label{theow0} If  $n\in L^{\infty}(D)$ with $\Im(n)>0$ (or $\Im(n)<0$) in a region $D_0\subset D$ with positive measure,  then there are no {\em real} transmission eigenvalues, hence such inhomogeneities always scatter.
\end{proposition} 
\begin{proof} 
We multiply the first equation in  (\ref{ten2})  by $\overline{u}$ and the second equation by $\overline{v}$, subtract the two and then integrate by parts to obtain
$$\int_D\left(-|\nabla u|^2\,dx+k^2n|u|^2\right) \, dx + \int_D\left(|\nabla v|^2\,dx-k^2|v|^2\right)\,dx=0$$
From the imaginary part we obtain that $u=0$ in $D_0$ and then by unique continuation it follows that $u=0$ in all of $D$. The conditions on the boundary imply that $v$ has zero Cauchy data hence $v=0$ in $D$ as well.
\end{proof}
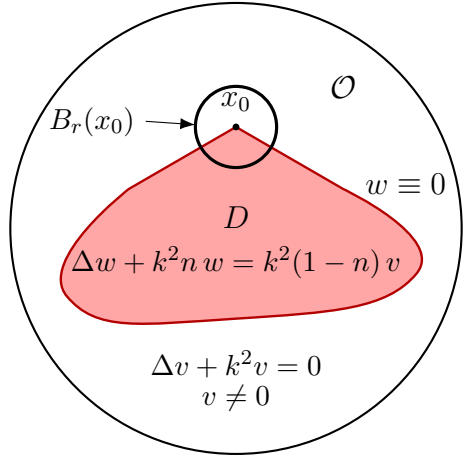
\begin{figure}[h!]
\centering
\begin{tikzpicture}[scale=1.5]

\def\Rout{2.0}     
\def\rball{0.36}   

\coordinate (O) at (0,0);        
\coordinate (X) at (0,0.90);     

\draw[thick] (O) circle (\Rout);

\node at (0.95,1.25) {$\mathcal{O}$};
\node at (0,-1.20) {{\small{$\Delta v+ k^{2}v =0$}}};
\node at (0,-1.50) {{\small{$v\neq 0$}}};
\coordinate (Rmid) at ( 0.95, 0.35);
\coordinate (R1)   at ( 1.60,-0.40);
\coordinate (Base) at ( 0.00,-0.80);
\coordinate (L1)   at (-1.50,-0.60);
\coordinate (Lmid) at (-0.95, 0.35);

\filldraw[fill=red!35, draw=red!70!black, line width=0.9pt]
  (X) -- plot[smooth, tension=1] coordinates { (Rmid) (R1) (Base) (L1) (Lmid) } -- cycle;

\node at (0,-0.30) {{\small{$\Delta w + k^{2} n\, w = k^{2}(1-n)\, v$}}};
\node at (0,0.10) {$D$};

\draw[very thick] (X) circle (\rball);
\fill (X) circle (0.03);
\node[above=2pt] at (X) {$x_0$};

\node[left] (BR) at ($(X)+(-0.80,0.05)$) {{\small{$B_r(x_0)$}}};
\draw[-{Latex[length=2mm]}] (BR.east) -- ($(X)+(-\rball,0.02)$);

\node[right] at (1.05,0.40) {$w \equiv 0$};

\end{tikzpicture}
\caption{Sketch of a non-scattering configuration.}
\label{fig:domain}
\end{figure}
\subsection{Singularities almost always scatter}
Knowing that real transmission eigenvalues exist, a natural question becomes: under what additional assumptions can these transmission eigenvalues be non-scattering wave numbers? In particular, when is the $v$ part of the transmission eigenfunction extendable outside as a (special) solution to the Helmholtz equation, so that (\ref{nonsn1})-(\ref{nonsn3}) holds. This question may be formulated slightly differently depending on the domain of definition one requires for this extension. If the domain of definition is required to be all of $\mathbb{R}^d$ (modulo a set of measure zero) one may also ask whether the $v$ part of the transmission eigenfunction could take the form of one of the special (physically common) incident waves. Our first approach to this question relies on viewing the boundary with vanishing Cauchy data as a free boundary and applying free boundary regularity results for second-order elliptic equations. The connection between the non-scattering property of a given inhomogeneity and the regularity of a free boundary was first introduced in \cite{CV} and \cite{fb7}. Here, we state the results and sketch the main ideas of the proof  following \cite{CV}.  These results  are stated in terms of sufficient conditions of non-smoothness of $\partial D$ for scattering to occur for a given incident wave. By negation they could as well have been stated as necessary smoothness properties of $\partial D$ that follow from non-scattering.   There is a striking similarity in the mathematical structure of the problem of non-scattering inhomogeneities, and the problem of domains that do not possess the  Schiffer or Pompeiu property \cite{pom0, pom2, pom1}.  Regularity properties of the latter were established by Williams \cite{pom3}, and the analysis\cite{CV} in several places borrows significantly from his original work.  In the formulation of our main results we refer to the region $D_\delta\subset D$, defined  by 
$$D_\delta:=\left\{x\in D, \; dist(x,\partial D)<\delta\right\}\;\; \;\;\; \mbox{for some fixed $\delta>0$}.$$
Again, for simplicity of notation,  we refer to  $n_D$ inside $D$ as $n$. 
\begin{theorem}\label{scatnon1} Let $k>0$ be a fixed wave number, let  $n$ in $L^\infty(D)$ be  the positive refractive index, and suppose that the boundary $\partial D$ is Lipschitz. Consider a nontrivial incident field $v$ satisfying (\ref{nonsn1}).  
\begin{enumerate}
\item Assume that $n\in C^{m, \mu}(\overline{D_\delta})\cap C^{1,1}(\overline{D_\delta})$  for $m \geq 1$, $0<\mu<1$, and there exists $x_0\in \partial D$ such that $(n(x_0)-1)v(x_0)\neq 0$. If  $\partial D\cap B_r(x_0)$  is not of class $C^{m+1, \mu}$ for any ball $B_r(x_0)$ of radius $r$ centered at $x_0$,  then the incident field $v$ is scattered by the inhomogeneity $(D, n)$. In other words: there exists no $H^2_0(D)$ solution to (\ref{nonsn2})-(\ref{nonsn3}).
\item Assume  that  $n$ is real analytic in $\overline{D_\delta}$, and there exists $x_0\in\partial D$   such $(n(x_0)-1)v(x_0)\neq 0$. If $\partial D\cap B_r(x_0)$ is not real analytic for any ball $B_r(x_0)$ of radius $r$ centered at $x_0$,  then the incident field $v$ is scattered by the inhomogeneity $(D, n)$. In other words: there exists no $H^2_0(D)$ solution to to (\ref{nonsn2})-(\ref{nonsn3}).
\end{enumerate}
\end{theorem}
\noindent
\begin{remark}
{\em The smoothness on $n$ in Theorem \ref{scatnon1} is only needed locally  in $\overline{D}\cap B_R(x_0)$.  In our particular application the solution $w$ is defined on all of $D$, but since the regularity of the free boundary is a local property, the results of Theorem \ref{scatnon1}  hold for $w$ solving  (\ref{nonsn2})-(\ref{nonsn3}) in $D \cap B_R(x_0)$ with zero Cauchy data only on $\partial D\cap B_R(x_0)$.}
\end{remark}
\noindent
If  the initial regularity of $\partial D$ is  $C^1$ and  $w\in C^2({\overline D}\cap B_R(x_0))$, Theorem \ref{scatnon1} is an immediate consequence of the regularity results  due to Kinderlehrer and Nirenberg  in  \cite[Theorem 1' on page 377]{pom000}.  However we initially assume that  $\partial D$ is only Lipschitz regular. Therefore,  we must first show that the free boundary $\partial D\cap B_R(x_0)$ is indeed $C^1$, and then verify that the (non-scattering) solution $w$ to (\ref{nonsn2})-(\ref{nonsn3}) is in $C^2({\overline D}\cap B_R(x_0))$.  This intermediate regularity is achieved with the help of a  classical result on regularity of the free boundary due to Caffarelli \cite[Section 1.2 and Theorem 3 on page 166]{pom00}, which we state in the following lemma, modified to the framework of our problem. This result refers to  a function $w$  that locally satisfies 
\begin{equation}\label{caf}
\Delta w=g  \quad \mbox{in } \, D\cap B_R(x_0), \qquad \mbox{such that $w=\displaystyle{\frac{\partial w}{\partial \nu}}=0 \quad \mbox{on } \, \partial D\cap B_R(x_0)$}.
\end{equation}
\begin{lemma}\label{th2} Suppose that $\partial D\cap B_R(x_0)$ is Lipschitz and $w$ satisfying (\ref{caf}) is in $C^{1,1}(\overline{D}\cap B_R(x_0))$. Furthermore, assume that  $w\leq 0$ in $D\cap B_R(x_0)$, and  $g$ has a $C^1$-extension $g^*$ in a neighborhood of  $\overline{D}\cap B_R(x_0)$  such that $g^*\leq -\alpha<0$. Then there exists $R'<R$ such that $\partial D\cap B_{R'}(x_0)$ is of class $C^1$ and $w \in C^2(\overline{D}\cap B_{R'}(x_0))$.
\end{lemma}
\noindent
In our case $g:=-k^2nw-k^2(n-1)v$ depends on the solution $w$. A major  obstacle to the application of Lemma \ref{th2} is to verify that the $H^2_0(D)$ solution $w$ to (\ref{nonsn2})-(\ref{nonsn3})  has all second derivatives  uniformly bounded in $D\cap B_R(x_0)$.  Our analysis is based on the study of a volume potential that expresses the solution $w$. We recall that   $n\in L^\infty(D) \cap C^{\alpha}(\overline D \cap B_R(x_0))$ for some ball $B_R(x_0)$ of radius $R$ centered at $x_0\in \partial D$ and any $0<\alpha\leq 1$.  We introduce the function 
$$
W(x) = \begin{cases} w(x) &\hbox{ for } x \in D~, \\ 
0 &\hbox{ for } x \in \mathbb{R}^d\setminus D.\end{cases}
$$ 
This function is in $H^2(\mathbb{R}^d)$ (since $w\in H^2_0(D)$) and for $m=2,3$ it follows from the Sobolev embedding theorem that $W \in C^{\alpha}(\mathbb{R}^d)$ for some $0<\alpha<1$.  Furthermore, $W$ solves
$$
\Delta W =\Psi \hbox{ in } \mathbb{R}^d, ~~ \hbox{ where }  \Psi(x) =\begin{cases} -\psi(x)&\hbox{ for } x \in D~ \\ 
0 &\hbox{ for } x \in \mathbb{R}^d\setminus D  \end{cases} 
$$
with $\psi(x)= k^2(n(x)-1)v(x)+k^2n(x)w(x),~x \in D$. The function $\psi$ is clearly in $L^\infty(D)$, and due to the assumptions about $n$ and $v$, and the $C^\alpha$ extendability of $w$, it has an extension that lies in $C^\alpha(B_r(x_0))$. If $\Upsilon(\cdot,\cdot)$ denotes the free space fundamental solution of the Laplace operator, we can express $W$ as a volume potential 
\begin{equation}\label{vol}
W(x)= W_\psi:=\int_D \psi(y) \Upsilon(x,y)\,dy
\end{equation} 
with $\psi= k^2(n-1)v+k^2nw \in L^\infty(D)\cap C^{\alpha}(B_r(x_0))$.   Standard regularity analysis  of the volume potential (\ref{vol}) with $\psi\in L^\infty(D)$, shows that 
$W_\psi \in C^1({\mathbb  R}^d)$ and its partial derivatives can be obtained by differentiation inside the integral. Due to the jump of $\Psi$ across $\partial D$, even for smooth $\psi$,  but with $\psi\neq 0$ on $\partial D$, the second derivatives of $W_\psi$ maybe become unbounded as $x$ approaches a boundary point from either inside or outside $D$.  Thus the volume potential is not necessarily in $C^2(\overline{D}\cap B_R(x_0))$ (for a general $\psi$). However, it is possible to show that the symmetric jumps of the second derivatives (the difference of the second derivatives taken at symmetric points on the normal from outside and inside $D$) are uniformly bounded near  $x_0\in \partial D$,   when $\psi \in C^{\alpha}(B_r(x_0))$ for some $0<\alpha<1$.  This result is proven in \cite[Lemma 4.2]{CV} following closely the proof of a similar result in \cite[Theorem 2]{pom3} for $\psi\equiv 1$.  Since $W=0$ outside $D$, this result  implies that all second derivatives of $W$ are uniformly bounded in $D \cap B_r(x_0)$ (for our special $\psi$). 

\noindent
Thus the above discussion implies  $w$ has an extension (by zero) in  $C^1(\mathbb{R}^d)$ and that $w$ is in $C^{1,1}(\overline{D}\cap B_r(x_0))$.  Lemma \ref{th2}  requires a real valued solution $w$. With this in mind, we note that the real valued function $w:=\Re(w)$ (we keep the same notation for the real part of $w$ for convenience) is a $H^2(D)$ solution to 
\begin{equation}\label{real1}
\Delta w+k^2nw=-k^2(n-1)\Re(v)~~\hbox{ with } w=\frac{\partial w}{\partial \nu}=0 \hbox{ on } \partial D~.
\end{equation}
and obviously the same regularity as above applies to the real part of $w$,  that is $w\in C^{1,1}(\overline D\cap B_r(x_0))$ and that $g= -k^2(nw+(n-1)\Re(v))$ has a $C^1$ extension to all of $\mathbb{R}^d$. Of course, one could also consider the imaginary part of the scattered field $w$, which satisfies the same equation as above with $\Re(v)$ replaced by $\Im(v)$. Accordingly, in what follows, everything holds true if we replace $\Re(v)$ by $\Im(v)$.  The essential and final missing step for application of Lemma \ref{th2} is to show that $w$ is of one sign. This is established in \cite[Proposition 5.2]{CV}, which we state below for the convenience of the reader.  Its technical proof follows almost verbatim the analysis by Williams in \cite[Section 5]{pom3} 
\begin{proposition}\label{prep2}
Assume that $\partial D$ is  Lipschitz,  $x_0\in \partial D$, and $n\in L^\infty(D)$. Furthermore, suppose $n$ lies in $C^{1,1}(\overline{D}\cap B_R(x_0))$, and  $(n(x_0)-1)\Re (v(x_0))\neq 0$.  Let  $w\in H^2(D)$ be a solution to (\ref{real1}). Then $w<0$ in $D\cap B_r(x_0)$ for some $0<r<R$ if $(n(x_0)-1)\Re (v(x_0))>0$, and $w>0$ in $D\cap B_r(x_0)$ for some $0<r<R$ if $(n(x_0)-1)\Re (v(x_0))<0$.
\end{proposition}
\noindent
With all the ingredients in place, we apply Lemma \ref{th2}, which under the assumptions of Theorem \ref{scatnon1} shows that if non-scattering occurs then $\partial D$ is locally $C^1$ and $w$ is locally $C^2$. This, in turn, enables the use of the Kinderlehrer-Nirenberg regularity theory to establish the regularity of non-scattering inhomogeneities expressed by  parts $1.$ and $2.$  of Theorem \ref{scatnon1}. We note that it is possible to start with weaker regularity assumptions on $\partial D$ than Lipschitz and still establish similar regularity results for non-scattering inhomogeneities. For details on such a generalization we refer the reader to \cite{fb7}.
\noindent
\begin{remark}\label{rem2}
{\em In our  non-scattering  application $v$ is real analytic in ${\mathcal O}\supseteq \overline D$  by virtue of being solution of the Helmholtz equation. However, it is clear from our analysis that the statements of Theorem \ref{scatnon1} is valid if $v$ is only defined on one side of $\partial D$, and  the regularity of $v$ up to the boundary matches that of $n$; simply notice that  our arguments rely only on the local regularity of the source term $(1-n)v$ in ${\overline D}\cap B_R(x_0)$.}
\end{remark}
\noindent
For $n$ sufficiently smooth and $\partial D$ Lipschitz, we already know that  (\ref{nonsn2})-(\ref{nonsn3}) has non-zero solution at a real transmission 
eigenvalue. In view of Remark \ref{rem2}, Theorem \ref{scatnon1}  provides a statement on up-to-the-boundary regularity (or the lack thereof)  of the $v$-part of the transmission eigenfunction. Indeed, we obtain the following result on the (ir)regularity as a simple consequence of  Theorem \ref{scatnon1}  
\noindent
\begin{corollary}\label{cor1}
Assume $k>0$ is a real transmission eigenvalue with eigenfunction $(u, v)$,  $\partial D$ is Lipshitz, $n\in L^\infty(D)$,  and there exits $x_0\in \partial D$ such that $n(x_0)-1\neq 0$. Then:
\begin{enumerate}
\item  If $n$ is real analytic in a neighborhood of $x_0$  and $\partial D\cap B_r(x_0)$ is not real analytic for any ball $B_r(x_0)$, then $v$ can not be real analytic  in any neighborhood of $x_0$, unless $v(x_0)=0$.
\item If $n\in C^{m, \mu}(\overline{D}\cap B_R(x_0))\cap C^{1,1}(\overline{D}\cap B_R(x_0))$ for $m \geq 1$, $0<\mu<1$ and some ball $B_R(x_0)$, and  $\partial D\cap B_r(x_0)$  is not of class $C^{m+1, \mu}$ for any ball $B_r(x_0)$, then $v$ cannot lie  in $C^{m, \mu}(\overline{D}\cap B_r(x_0))\cap C^{1,1}(\overline{D}\cap B_r(x_0))$ for any ball $B_r(x_0)$, unless $v(x_0)=0$.
\end{enumerate}
\end{corollary}
\noindent

\subsubsection{Non-vanishing condition and inhomogeneities with corner singularities} 
The non-vanishing condition $(n(x_0)-1)v(x_0)\neq 0$ at the boundary point $x_0$ is essential for all free boundary regularity techniques. The condition on the contrast $n-1$ is very natural, and such an assumption is already required in the study of the transmission eigenvalue problem. Several physical incident waves such as 
$$v(x):=e^{ikx\cdot \eta} \qquad \mbox{a plane wave with incident direction $\eta$, \qquad \qquad or }$$
$$v(x):=\Phi(x;z) \qquad \mbox{a point source located at  $z\in {\mathbb R}^d\setminus{\overline{\mathcal O}}$},$$
where $\Phi(\cdot ; \cdot)$ is the free space fundamental solution of the Helmholtz equation always satisfy the non-vanishing condition $v(x_0)\neq 0$. However, for many other incident waves, such as Herglotz wave functions given by (\ref{herg}), this non-vanishing condition is not always satisfied. 

\noindent
When the support $D$ of the inhomogeneity contains a circular conical point, a vertex, or an edge, the set of non-scattering wave numbers is empty, provided only the non-vanishing condition on $n-1$ is satisfied. This result was first proven by  Bl{\aa}sten, P\"aiv\"arinta  and Sylvester in \cite{nonscat} for a  rectangular corner  in ${\mathbb R}^d$, $d\geq 2$,  and later extended in \cite{PSV17} for a convex corner, using the so-called complex geometric optics (CGO) solutions for the Helmholtz equation. A further development, based on new types of CGO solutions introduced in \cite{xiao}, removes the convexity assumption for corners in ${\mathbb R}^2$. This approach begins with the transmission eigenvalue problem for a real transmission eigenvalue and shows that the component  $v$ of the eigenfunctions cannot be real analytic in a neighborhood of the vertex. In particular, if $x_0\in \partial D$ it is straightforward to show that for transmission eigenfunctions  $(w,v)$ satisfying (\ref{ten}), and any solution to
$$\Delta \psi +k^2n\psi=0 \qquad \mbox{in $B_\epsilon(x_0)\cap D$,}$$
we have that 
$$\int_{\partial B_\epsilon(x_0)\cap D}\left(\frac{\partial w}{\partial \nu}\psi-\frac{\partial \psi}{\partial \nu}w \right)\, ds=\int_{B_\epsilon(x_0)\cap D}k^2(1-n)v\psi\, dx,$$
where $\partial B_\epsilon(x_0)\cap D$ is the part of the sphere inside $D$. Choosing test functions $\psi$ from a family of CGO  solutions that decay rapidly away from $x_0$, and combining this with asymptotic analysis, one arrives at a contradiction if $v$ is assumed to be analytic in $B_\epsilon(x_0)$ and $x_0$ is a boundary singular point.

\noindent
The most comprehensive analysis of geometries with boundary singularities is due to Elschner and Hu  \cite{ElH18},  who employed a refined singularity analysis of solutions to (\ref{nonsn2})-(\ref{nonsn3}) in a neighborhood of the boundary singular points. In their work, the non-vanishing assumption on $n-1$ at the boundary singularity is also relaxed.  In the following we state precisely their result  and refer the reader to \cite{ElH18} for the proof.  
\begin{theorem} \label{non-scati} Let $(D,n_D)$  be an inhomogeneity with positive refractive index $n_D\in L^\infty(D)$, and let $n$ denote the function which equals $n_D$ in $D$ and the constant $1$ in $\mathbb{R}^d\setminus D$, $d=2,3$. Assume that there exists a boundary point $x_0\in \partial D$ that is a planar corner in ${\mathbb R}^2$, or an edge or a circular conic point  in ${\mathbb R}^3$(see \cite[Definitions 2, 3]{ElH18}  for the precise definition of these boundary singularities). Furthermore, assume that there exists $m\in {\mathbb N}_0$, $\mu\in (0,1)$  and some  $R>0$ such that  $n\in C^{m, \mu}(\overline{D\cap B_R(x_0)})\cap W^{m,\infty}(B_R(x_0))$ and  $\nabla^m(n-1)(x_0)\neq 0$. Then this inhomogeneity scatters every incident wave non-trivially.
\end{theorem}
\noindent

\subsection{Remarks on analytic inhomogeneities}\label{anal}
Theorem \ref{scatnon1}  asserts that, if the index of refraction 
$n_D$ is real analytic, then any non-degenerate incident wave (one for which $(n_D-1)v(x_0)\neq 1$) will scatter  from any inhomogeneity, whose boundary is not real analytic. This raises the natural question: what happens for real analytic inhomogeneities? This question has been analyzed  in some detail for $n_D=\hbox{constant}\neq 1$ in two dimensions. The analysis is found in \cite{HV,HV2}, and the techniques used there are of a variational nature and generalizations of techniques, that have already been successfully used to investigate whether a domain has the Schiffer or Pompeiu property (see \cite{BrKaTa,GaSe}). As a first example example consider incident plane waves 
$e^{ik x\cdot \eta}$, $\eta \in \mathcal{S}^1$. We introduce the notion of the width  of an inhomogeneity $D$ in the direction $\lambda \in \mathcal{S}^1$
$$
w_D(\lambda)= \sup_{x\in D} \lambda\cdot x - \inf_{x\in D} \lambda \cdot x~,
$$
and the associate notions of maximal width and minimal width
$$
W^*(D)=\max_{\lambda \in \mathcal{S}^1} w_D(\lambda) \quad  \hbox{ and } \quad w_*(D)=\min_{\lambda \in \mathcal{S}^1} w_D(\lambda) 
$$
It is among other results proven that
\begin{enumerate}
\item Any inhomogeneity $D$  with constant index of refraction  $\eta_D<1$ will scatter any incident plane wave.
\item Any strictly convex inhomogeneity $D$ for which $W^*(D)\ge 2 w_*(D)$  will scatter any plane wave.
\end{enumerate}
Since $(n_D-1)e^{ik x_0\cdot \eta}$ is always different from zero, the novelty of these results is only for real analytic inhomogeneities. We refer to \cite{HV} for details and other results about the scattering of plane waves. In \cite{HV2}  the reader can find the analysis for more general incident waves and a wide class of real analytic and piecewise real analytic inhomogeneities including (all non circular) ellipses, certain non convex domains, and domains with inward cusps. For any inhomogeneity $D$ in this class of domains  the authors  established the existence of a point $z_0 \in \mathbb{C}^2$ with the property that any incident wave extendable as an analytic function to a $\mathbb{C}^2$ neighborhood of $\R^2$, containing $z_0$, and non-vanishing at $z_0$, will be scattered by $D$. In particular any incident plane wave will also scatter from such an inhomogeneity. \footnote{Inhomogeneities with inward cusps are of interest, since this kind of non-Lipschitz singularity was not excluded by the regularity result in \cite{fb2}. For our example the point $z_0$ coincides with the cusp.} Finally we also want to call attention to \cite{vx} and \cite{vx2} where it is shown that any Herglotz wave with a fixed density $\varphi$ is at most non-scattering at {\it finitely } many wave numbers for inhomogeneities in the shape of quite general non spherically symmetric perturbations of a disk.

\section{The Case of $A\not\equiv I$.}
\label{AneqI}

Let $D$ denote a bounded, simply connected inhomogeneity. As before
$$
A= \begin{cases} A_D \hbox{ in } D \\
I  \hbox{ in } \R^d \setminus D \end{cases} \hbox{ and } n= \begin{cases} n_D \hbox{ in } D \\
1  \hbox{ in } \R^d \setminus D \end{cases}
$$
Here, $A_D$ and $n_D$, the electromagnetic constitutive parameters of satisfy \eqref{eq:Aellip} and $n_D >0$, respectively. By an incident wave, $u^{\In}$, at wavenumber $k$, we shall understand a solution to $(\Delta+k^2)u^{\In}=0$.  We recall (cf. (\ref{MainGov1})) that when there is contrast in both electromagnetic properties then non-scattering for the incident wave occurs exactly when there exists a solution to 
\begin{equation}
\label{nonscat}
\nabla \cdot(A \nabla u) +k^2 nu=\nabla \cdot((I-A)\nabla u^{\In})+k^2(1-n)u^{\In} ~~\hbox{ in } ~~\R^d,
\end{equation}
with $u=0$ in $\R^d \setminus D$. When $\partial D$ is of class $C^{1,\alpha}$, this is equivalent to the existence of a solution to
\begin{equation}
\label{nonscat2}
\nabla \cdot(A_D \nabla u) +k^2 n_D u=\nabla \cdot((I-A_D)\nabla u^{\In})+k^2(1-n_D)u^{\In} ~~\hbox{ in } ~~D,
\end{equation}
with 
\begin{equation}
\label{nonscatbc}
u=0 \hbox{ and } \nu \cdot A_D \nabla u = \nu \cdot (I-A_D)\nabla u^{\In} \hbox{ on }\partial D. 
\end{equation}
As pointed out earlier a necessary condition for the existence of an incident wave at wave number $k$, which is non-scattering. is that $k$ be a transmission eigenvalue. We refer to (\ref{te}) or (\ref{te2})  for the formulation of the transmission eigenvalue problem in the case of contrast in $A$ (as well as $n$). In brief: under very general conditions there exists a countable infinite set of real transmission eigenvalues, even when the boundary is only assumed to be Lipschitz. More precisely for  anisotropic media with contrast only in $A\in L^\infty({\mathbb R}^d)$, i.e. $n\equiv1$,  the existence of infinitely many real transmission eigenvalues accumulating only at $+\infty$ is proven in \cite{1}, provided that the matrix $A_D-I$ is either positive definite or negative definite uniformly in $D$. In this case it is possible to use a fourth order approach similar to the one discussed in Section \ref{TEn}.  The case  with contrast in both $A$ and $n$ in $L^\infty({\mathbb R}^d)$, is discussed in \cite{CK}, where it is shown that an infinite set of real transmission eigenvalues exists provided that $A_D-I$ and $n_D-1$ are of one sign (the same or opposite) uniformly in $D$. The approach in this case differs from the one discussed in  Section \ref{TEn}, since it is not possible to write the transmission eigenvalue problem in terms of  a fourth order differential operator. Under the above assumptions it is also known that all transmission eigenvalues (real and complex) form a discrete set with only infinity as accumulation point. The state-of the-art of the spectral analysis for the transmission eigenvalue problem (\ref{te2}), including discreteness of real and complex eigenvalues, completeness of generalized eigenfunctions in $(L^2(D))^2$, and Weyl asymptotics for the eigenvalue counting function, can be found in  \cite{nguyen2, nguyen}. This spectral analysis is conditional on some 
``ellipticity''  assumptions on the coefficients at  the boundary $\partial D$. More specifically, it requires that $\partial D$ is of class $C^2$, and that $A_D$ and $n_D$ are continuous on $\overline D$ and satisfy  for all $x\in \partial D$
\begin{equation}\label{dis}
(A_D(x)\nu \cdot \nu)(A_D(x)\tau \cdot \tau)-(A_D(x)\nu \cdot \tau)^2\neq
 1 \qquad \mbox{and} \qquad (A_D(x)\nu \cdot  \nu)n_D(x)\neq 1
 \end{equation}
 for all unit vectors $\tau\in {\mathbb R}^d$  perpendicular to the unit normal vector $\nu$ (the first condition is equivalent to the complementing condition, due to Agmon, Douglis and Nirenberg \cite{ADN}). We refer the reader to \cite{vodev} for results about the location  of transmission eigenvalues in the complex plane. For a comprehensive discussion we also refer the reader to \cite[Chaper 4]{CakoniColtonHaddar2016}.

\medskip
\noindent
When it comes to wave numbers which possess non-scattering incident waves the situation is different; just as in the case when there is no contrast in the principal term, the existence of such a wave number implies some regularity on the boundary of the inhomogeneity, $\partial D$. Analogous to the case $A\equiv I$ we note that, if $D:=B_R$ is a ball of radius $R$ centered at the origin, and $A_D:=a_D(r)I$, $n_D:=n_D(r)$ with scalar functions  depending only on the radial variable $r$, satisfying
$$\frac{1}{R}\int_0^R\left(\frac{n_D(r)}{a_D(r)}\right)^{1/2} \,dr \neq 1~,$$
then it is possible to demonstrate, by separation of variables, the existence of infinitely many non-scattering wave numbers \cite{CK, coltonkress}.  In fact, for this spherically stratified medium the set of real transmission eigenvalues and the set of non-scattering wave numbers coincide. Furthermore, the non-scattering incident waves are superpositions of plane waves, otherwise known as Herglotz wave functions with particular densities, and each density is associated with an infinite set of non-scattering wave numbers.

\noindent
The following result was proven in \cite{CVX}; it ensures regularity of $\partial D$ near any point $x_0$ on $\partial D$ where the following non-degeneracy condition is satisfied
\begin{equation}
\label{nondeg}
\nu \cdot (A_D-I)\nabla u^{\In}(x_0)\neq 0
\end{equation}
\begin{theorem}
\label{anisoreg}
Let $D$ be a bounded domain in $\R^d$, $d\ge2$, of class $C^{1,\mu}$. Suppose for some integer $m\ge 1$ that $n_D\in C^{m,\mu}(\overline{D})$  and that $A_D\in \left(C^{m+1,\mu}(\overline{D})\right)^{d\times d}$ with the condition \eqref{eq:Aellip} satisfied. 
Suppose there exists a solution $u$ to (\ref{nonscat}) with $u=0$ in $\R^d \setminus D$ and $u^{\In}$ a solution to $(\Delta+k^2) u^{\In}=0$ in $\R^d$. Let $x_0$ be a point on $\partial D$ at which the non-degeneracy condition (\ref{nondeg}) is satisfied.
	Then $\partial D$ is of class $C^{m+1,\mu}$ near $x_0$.
	Moreover, if both $A_D$ and $n_D$ are $C^{\infty}$ on $\overline{D}$, then $\partial D$ is $C^{\infty}$ near $x_0$. If $A_D$ and $q_D$ are real analytic on $\overline{D}$, then $\partial D$ is also analytic near $x_0$.
\end{theorem}
\noindent
The proof of this result is truly local in character and so the regularity requirements on $A_D$ and $n_D$ are only needed near $x_0$ and the incident wave $u^{\In}$ need only be defined in a neighborhood of $\partial \Omega$ near $x_0$. The initial assumption about the boundary of $D$ being of class $C^{1,\mu}$ can be relaxed to the assumption that it be only Lipschitz \cite{fb4}. In that case the non-degeneracy condition has to be replaced by 
\begin{eqnarray}
\label{nondeg2}
\nu \cdot (A_D-I)\nabla u^{\In}>c> 0 \hbox{ in a neighborhood of } x_0 \hbox{ or } \\
 \nu \cdot (A_D-I)\nabla u^{\In}<-c<0 \hbox{ in a neighborhood of } x_0 ~. \nonumber
\end{eqnarray}
\noindent
In the case of contrast only in the coefficient $n$, the non-degeneracy condition in Theorem \ref{scatnon1} naturally separates into a condition on $n_D$: $(n_D-1)(x_0)\neq 0$, {\bf and } one on $u^{\In}$: $u^{\In}(x_0) \neq 0$.
The non-degeneracy condition (\ref{nondeg}) can also be interpreted as two separate conditions, namely that $(A_D-I)\nu(x_0) \neq 0$, {\bf and } that $\nabla u^{\In}(x_0)$ is not orthogonal to $(A_D-I)\nu(x_0) $. Similarly for (\ref{nondeg2}), it requires that $(A_D-I)\nu $ stays bounded away from zero near $x_0$ and that $\nabla u^{\In}$ is non-zero and directed in the same direction as, or the opposite direction of, $(A_D-I)\nu$ near $x_0$. It is clear that the presence of a contrast in $n$ or $A$ is necessary for the regularity results in Theorem \ref{scatnon1} and Theorem \ref{anisoreg}, however in the case of Theorem \ref{scatnon1} we noted that the condition $u^{\In}(P)\neq 0$ was not necessary to rule out corners (vertices or conical points) and the nondegeneracy condition on $(n_D-1)$ may be relaxed.  The situation is completely different in the case of Theorem \ref{anisoreg}, where some non-degeneracy of $\nabla u^{\In}$ is essential. This is illustrated by the following example.

\vskip 10pt
\noindent
{\bf Example}: Consider the case $D=(0,1)^2$, with $A_D=aI$ and $n_D=a$ for some constant $a\neq 1$. If one takes $u^{\In}= \cos (k_1\pi x_1) \cos (k_2 \pi x_2)$ for some pair $(k_1,k_2) \in \mathbb{Z}^2$, then $u=0$ is a solution to (\ref{nonscat}) for the wave number $k=\pi \sqrt{k_1^2+k_2^2}$ with $u=0$ in $\R^2\setminus D$, but $\partial D$ is not smooth. We note that $\nabla u^{\In}=0$ at the four corners of $D$, and actually $(a-1)\nabla u^{\In}\cdot \nu =0$ at all other boundary points, so that the non-degeneracy condition (\ref{nondeg2}) is violated; we also note that $u^{\In}\neq 0$ at all four corners, and so $(n_D-1)u^{\In} \neq 0$ at all four corners, but this non-degenracy condition on the lower order term is not sufficient to guarantee boundary regularity when there is also contrast in the principal order coefficient.

\noindent
For this inhomogeneity one could also take $u^{\In}=\sin (k_1 \pi x_1)\sin (k_2 \pi x_2)$ with $k_i \in \mathbb{Z}\setminus \{0\}$. Then the function defined by
$$
u(x_1,x_2) = \begin{cases}\frac{1-a}a \sin (k_1 \pi x_1)\sin (k_2 \pi x_2) \hbox{ for } (x_1,x_2) \in D \\ 0 \hbox{ for } (x_1,x_2) \in \mathbb{R}^2\setminus D \end{cases}
$$
is a solution to (\ref{nonscat}) for $k=\pi \sqrt{k_1^2+k_2^2}$ and with $u=0$ in $\R^2 \setminus D$. We note that $\nabla u^{\In}=0$ at the four corners of $D$, and in this case $(a-1)\nabla u^{\In}\cdot \nu \neq 0$ in a neighborhood of all four corners, but this quantity approaches $0$ as one approaches any of the corners. In other words the non-degeneracy condition (\ref{nondeg2}) is again violated; in this case we also have $u^{\In}=0$ on $\partial D$.

\vskip 5pt
\noindent
Scattering properties of an (isotropic) inhomogeneity $(A_D,n_D,D)$ in ${\mathbb R}^2$ with the support $D$ containing a corner are studied in \cite{nn5, xiao} using the CGO solutions approach. We summarize the main results in the remark below.
\begin{remark}\label{jingni}
{\em Let $x_0\in \partial D$ be the vertex of a corner with aperture $2\theta$, $\theta\in(0,\pi/2)\cup(\pi/2,\pi)$. Assume that $A_D:=aI$ with a scalar function $a\in C^{1,\alpha}(\overline{D}\cap B_R(x_0))$, for some ball $B_R(x_0)$, and with $a(x_0)\neq 1$. Then any incident field $u^{\In}$ is scattered, provided that
\begin{itemize}
\item [(a)] $u^{\In}(x_0)\neq 0$ and $\nabla u^{\In}(x_0)\neq 0$, or
\item[(b)]  $u^{\In}(x_0)=0$, with  $2\theta\neq l \pi/N$ for any integer $l\geq 1$, with $N$ being the vanishing order of $u^{\In}$ at $x_0$ (in particular if $2\theta\in (0,2\pi)/{\mathbb Q}\pi)$), or 
\item [(c)] $u^{\In}(x_0)\neq 0$ and $\nabla u^{\In}(x_0)= 0$, and  either $2\theta \notin \left\{\pi/2, 3\pi/2\right\}$  or $2\theta \in \left\{\pi/2, 3\pi/2\right\}$ and $a(x_0)\neq n(x_0)$.
\end{itemize}
}
\end{remark} 
\noindent
If the assumptions of Remark \ref{jingni} are not satisfied, the analysis is inconclusive. The above example provides a non-scattering inhomogeneity with corners (and incident waves) falling within the inconclusive cases.

\vskip 10pt
\noindent
The proof of Theorem \ref{anisoreg} borrows a lot from the classical proof of regularity of free boundaries found in \cite{pom000}, in particular the use of the Hodograph transform, in which the solution to the over-determined boundary value problem in $D$ (\ref{nonscat2}) with boundary conditions  (\ref{nonscatbc})) is chosen as one coordinate, while the other $d-1$ coordinates are kept unchanged. By the Hodograph transform $D$ thus (locally) gets mapped to a domain with a flat boundary, and the regularity of the inverse of this transform determines the regularity of the boundary of $D$. This inverse mapping also has the form where one variable is transformed and the other $d-1$ stay unchanged. The technical part of the Hodograph approach now consists in deriving a boundary value problem satisfied by the inverse-transformation function. Not surprisingly it turns out to be a (non-linear) second order elliptic problem with a first order boundary condition on the transformed domain  (with a flat boundary). The fact that one may determine a boundary condition stems from the fact that the problem for $u$ is over-determined (has two boundary conditions) -- the first ($u=0$) was used to ``flatten" the boundary, the second results in a boundary condition. The novelty of the result in Theorem \ref{anisoreg} is the nonstandard right hand side and boundary condition, which originates from the non-scattering context.

\noindent
In the example above we provided a simple isotropic example of non-scattering for a non-smooth domain and particular incident waves, when the non-degeneracy condition (\ref{nondeg2}) is not satisfied. In the anisotropic situation there is a quite natural way to construct conductivity and index of refraction pairs $A_D,n_D$ that never scatter, {\it i.e.}, that don't scatter {\bf any } incident wave at {\bf any} frequency (see \cite{CVX}). It proceeds by defining
$$
A_D=\Phi^*[I]= \frac{D\Phi D\Phi^T}{\hbox{Det} D\Phi}\circ \Phi^{-1} \hbox{ and } n_D = \Phi^*[1]=\frac{1}{\hbox{Det} D\Phi}\circ \Phi^{-1}
$$
where $\Phi$ is a diffeomorphism of $D$ onto $D$, with $\Phi(x)=x$ for $x\in \partial D$ (supposing $\hbox{Det}D\Phi$ is positive). The construct $\Phi^*[I]$ is referred to as the push-forward of the identity ($I$) by the diffeomorphism $\Phi$. We note that this construction necessarily gives rise to a conductivity matrix $\Phi^*[I]$ that is anisotropic (has at least two different eigenvalues) unless $\Phi$ is the identity mapping. Indeed, if $\Phi^*[I]$ has only one eigenvalue then it will be a scalar and then, for $d=2$, the Cauchy-Riemann equations, or in general, Liouville's Theorem , guarantee that $\Phi$ is the identity everywhere, and not just on the boundary. Since this construction gives rise to non-scattering at any wave number it yields an inhomogeneity $(A_D,n_D,D)$ for which the transmission spectrum is all of $\mathbb{C}$. Fortunately it is not hard so see
$$
<\Phi^*[I]\nu,\nu> <\Phi^*[I]\tau,\tau> - <\Phi^*[I]\nu,\tau>^2=1~~ \hbox{ on } \partial D~,
$$
where $\nu$ denotes a unit normal vector and $\tau$ denotes any unit tangent vector (cf. \cite{GV}), and so this construction does not contradict the discreteness result for the transmission spectrum discussed earlier (see (\ref{dis})). This construction can be made on any Lipschitz  domain $D$, and in that context it becomes interesting to understand why it does not contradict the regularity result from Theorem \ref{anisoreg}. Loosely speaking: for $A_D=\Phi^*[I]$ and $n_D=\Phi^*[1]$ to satisfy the regularity assumptions of Theorem \ref{anisoreg}, the boundary $\partial D$ must already have the regularity asserted by the theorem. To be more precise one has the following result (cf. \cite{GV})

\begin{theorem}
\label{GV}
Let $D$ be a bounded domain in $\R^d$, $d\ge2$, of class $C^{1,\mu}$. Let $x_0$ be a point on $\partial D$ with $\nu$ denoting the outward unit normal. Suppose the diffeomorphism $\Phi$ satisfies
\begin{equation}
\label{nondeg5}
\nu \cdot (\Phi^*[I]-I)\xi \neq 0 \hbox{ for  some } \xi \in \mathcal{S}^{d-1}\hbox{ at the point } x_0~.
\end{equation}
Suppose furthermore that $\Phi^*[I](\cdot)$ is $C^{m+1,\mu}$ near $x_0$ in $\overline{D}$ and $\hbox{Det} D\Phi\circ \Phi^{-1}$ is $C^{m,\mu}$ near $x_0$ in $\overline{D}$  for some $m\ge 1$. Then $\partial D$ is of class $C^{m+1,\mu}$ near $x_0$.
\end{theorem}

\vskip 10pt
\begin{remark}
{\em The condition (\ref{nondeg5}) is equivalent to $(D\Phi - I)\nu \neq 0$ at $x_0$. The requirement that $\hbox{Det} D\Phi\circ \Phi^{-1}$ be $C^{m,\mu}$ near $x_0$ is not essential -- it suffices to require that  $\Phi^*[I](\cdot)$ be $C^{m+1,\mu}$ near $x_0$. For $d\ge 3$ this is particularly obvious, due to the formula $\hbox{Det} \Phi^*[I] = (\hbox{Det} D\Phi)^{2-d}\circ \Phi^{-1}$.
}
\end{remark}
\vskip 10pt
\noindent
If we strengthen the non-degeneracy condition of Theorem \ref{GV} to assert that $(D\Phi - I)\nu$ has a nonzero  tangential component at $x_0$ then the proof of Theorem \ref{GV} is very elementary. To illustrate this let us consider the case $d=2$ and assume $x_0=(0,0)$ with the $\partial D$ locally given by 
$x_2=\psi(x_1)$ for some $C^{1,\mu}$ function $\psi$ with $\psi'(0)=0$. Since the mapping $x_1 \rightarrow (x_1,\psi(x_1))$ is $C^{1,\mu}$ near $0$ and since $D\Phi D\Phi^T\circ \Phi^{-1}= \hbox{Det}D\Phi \circ \Phi^{-1} \Phi^*[I]$ is $C^{m,\mu}$ on $\overline D$ it follows that the mapping $x_1 \rightarrow D\Phi D\Phi^T(x_1,\psi(x_1))$ is $C^{1,\mu}$ near zero. A simple calculation now gives that
\begin{eqnarray*}
&&D\Phi D\Phi^T [x_1,\psi(x_1)] = \\
&&\left[\begin{array}{lll} [1-(\phi_1)_2 \psi']^2+ [(\phi_1)_2]^2 &(1-(\phi_1)_2\psi')(1-(\phi_2)_2)\psi'+(\phi_1)_2(\phi_2)_2   \\
(1-(\phi_1)_2\psi')(1-(\phi_2)_2)\psi'+&\hskip -8pt(\phi_1)_2(\phi_2)_2 \hskip 16pt \left(1- (\phi_2)_2\right)^2 (\psi')^2+[(\phi_2)_2]^2 
\end{array}\right][x_1,\psi(x_1)]
\end{eqnarray*} 
and it thus follows that the mapping
\begin{equation}
\label{Ck1}
x_1 \rightarrow \left[ \begin{array}{ll} [1-(\phi_1)_2 \psi']^2+ [(\phi_1)_2]^2 \\
(1-(\phi_1)_2\psi')(1-(\phi_2)_2)\psi'+(\phi_1)_2(\phi_2)_2 \\
\left(1- (\phi_2)_2\right)^2 (\psi')^2+[(\phi_2)_2]^2
 \end{array}    \right][x_1,\psi(x_1)]
\end{equation}
is $C^{1,\mu}$ near $0$. Suppose the coordinate system has been chosen so that $\nu=(0,1)$ and $\psi'(0)=0$. From the stricter non-degeneracy condition that $(D\Phi-I)\nu$ has a tangential component at $x_0$, it now follows that $(\phi_1)_2(0,0) \neq 0$. From the fact that $\Phi$ is a diffeomorphism it also follows that $(\phi_2)_2(0,0)=\hbox{Det}D\Phi(0,0) \neq 0$.  Consider the mapping $F: \mathbb{R}^3 \rightarrow \mathbb{R}^3$ given by 
$$ 
F(a,b,c)=((1-ba)^2 + b^2, (1-ba)(1-c)a+bc,(1-c)^2a^2+c^2)^T
$$
locally near $(0,b,c)$, with $b=(\phi_1)_2(0,0)\neq 0$ and $c=(\phi_2)_2(0,0) \neq 0$. With this notation the mapping (\ref{Ck1}) may be written
\begin{equation}
\label{Ck2}
x_1 \rightarrow F(\psi'(x_1),(\phi_1)_2(x_1,\psi(x_1)),(\phi_2)_2(x_1,\psi(x_1)))
\end{equation}
We compute
$$
\hbox{Det}DF|_{(0,b,c)}=-4cb~,
$$ 
so that $F$ locally (near $(0,(\phi_1)_2(0,0),(\phi_2)_2(0,0)))$ has a $C^\infty$ inverse. By composing (\ref{Ck2})({\it i.e.}, (\ref{Ck1})) with $F^{-1}$ we conclude that the mapping $x_1 \rightarrow (\psi'(x_1),(\phi_1)_2(x_1,\psi(x_1)),(\phi_2)_2(x_1,\psi(x_1)))$ is $C^{1,\mu}$ near $0$. In particular $x_1 \rightarrow \psi(x_1)$ is $C^{2,\mu}$ near $0$. By repeating this argument another $m-1$ times (with increasingly regular $\psi$) we finally arrive at the conclusion that $\psi$ is in $C^{m+1,\mu}$ near $0$, {\it i.e.}, $\partial D$ is of class $C^{m+1,\alpha}$ near $x_0$. For more details about the higher dimensional version and in particular about the proof in the case when $D\Phi \nu_{x_0} =\alpha \nu_{x_0}$, with $\alpha \neq 1$ we refer to \cite{GV}. This latter case is more complicated since it apparently  cannot be resolved by only looking at the boundary values of $\Phi^*[I]$. Instead one formulates a second order overdetermined elliptic boundary value problem in divergence form with cofficient $\Phi^*[I]$ and then proceeds with an application of the Hodograph Transform (as in \cite{pom000}) to establish the ``free boundary" regularity result for $\partial D$. To be more precise one defines $w(y) = \nu_{x_0} \cdot \Phi^{-1}(y) -\nu_{x_0} \cdot y$, where $\nu_{x_0}$ denotes the normal to $\partial D$ at $x_0$. This function satisfies
$$
\nabla \cdot (\Phi^*[I] \nabla w) =-\nabla \cdot \Phi^*[I]\nu_{x_0}  ~~ \hbox{ in } D
$$
with 
$$
w=0 \hbox{ and } \nu \cdot \Phi^*[I]\nabla w =\nu \cdot (I-\Phi^*[I])\nu_{x_0} ~~ \hbox{ on } \partial D.
$$
One now uses this $w$, locally near $x_0$, as a coordinate (in the $\nu_{x_0}$ direction) for the Hodograph Transform. For more details we refer to \cite{GV}.

\section*{Acknowledgments}
{The research of FC was partially supported  by the NSF Grant DMS-24-06313. The research of MSV was partially supported by NSF Grant DMS-22-05912.}

\end{document}